\documentclass[a4paper, 11pt]{article}
\usepackage{graphicx} 
\usepackage[all]{xy}
\usepackage{amsfonts, amssymb, amsmath,url}
\usepackage{stmaryrd}
\usepackage{amsthm}
\usepackage{epsfig}
\usepackage{ifthen}

\newtheorem{Thm}{Theorem}
\newtheorem{Lem}{Lemma}
\newtheorem{Prop}{Proposition}

\newtheorem{Coro}{Corollary}

\theoremstyle{definition}
\newtheorem{Rem}{Remark}

\newtheorem{Cond}{Condition}

\newtheorem{Example}{Example}

\newcommand{\comment}[1]{}
\newcommand{\ind}{{\bf 1}}
\def\indd#1{{\bf 1}_{\{#1\}}}
\newcommand{\proba}{\mathbb P}
\newcommand{\esp}{{\mathbb E}}

\newcommand{\pp}{^\prime}

\newcommand{\defe}{\mathrel{\mathop:}=}

\newcommand{\filF}{{\cal F}}

\newcommand{\calN}{{\cal N}}

\def\indzd#1{\{#1_t\}_{t\in\mathbb Z^d}}

\def\indn#1{\{#1_n\}_{n\in \mathbb N}}

\newcommand{\eqnh}{\begin{eqnarray*}}
\newcommand{\eqne}{\end{eqnarray*}}
\newcommand{\eqnhn}{\begin{eqnarray}}
\newcommand{\eqnen}{\end{eqnarray}}
\newcommand{\equh}{\begin{equation}}
\newcommand{\eque}{\end{equation}}

\def\prodd#1#2#3{\prod_{#1 = #2}^{#3}}

\def\sif#1#2{\sum_{#1=#2}^\infty}

\newcommand{\widebar}{\overline}

\def\nn#1{{\left\|#1\right\|}}

\def\bnn#1{\Big\|#1\Big\|}

\def\sabs#1{|#1|}
\def\babs#1{\Big|#1\Big|}

\def\bccbb#1{\Big\{#1\Big\}}

\def\bpp#1{\Big(#1\Big)}
\def\spp#1{(#1)}

\def\sbb#1{[#1]}
\def\bbb#1{\Big[#1\Big]}

\def\floor#1{\left\lfloor #1 \right\rfloor}
\def\sfloor#1{\lfloor #1 \rfloor}

\def\vv#1{{\bf #1}}

\def\d{{\rm d}}

\def\mand{\mbox{ and }}
\def\qmand{\quad\mbox{ and }\quad}

\def\mwith{\mbox{ with }}

\def\mfa{\mbox{ for all }}

\def\mmas{\mbox{ as }}
\def\qmmas{\quad\mbox{ as }\quad}

\def\adaptF#1{\{#1_t,\filF_t:0\leq t<\infty\}}

\def\probato{\stackrel{\proba}{\longrightarrow}}

\def\weakto{\Rightarrow}

\def\limn{\lim_{n\to\infty}}
\def\limm{\lim_{m\to\infty}}

\def\wt#1{\widetilde{#1}}
\def\wb#1{\widebar{#1}}

\def\Z{{\mathbb Z}}
\def\Zd{{\mathbb Z^d}}
\def\R{{\mathbb R}}

\def\N{{\mathbb N}}

\def\ifhead#1#2{\left\{\begin{array}{#1@{\quad\mbox{ if }\quad}#2}}
\def\ifend{\end{array}\right.}

\def\bbr#1{\llbracket#1\rrbracket}
\def\lleq{\preceq}
\def\ggeq{\succeq}
\def\indzd#1{\{#1_i\}_{i\in\Zd}}

\def\Zd{{\mathbb Z^d}}

\def\N{\mathbb N}

\def\mfa{\mbox{ for all }}

\title{On the asymptotic normality of kernel density estimators for linear random fields}% by $m$-approximation}
\author{Yizao Wang and Michael Woodroofe\\ \\Department of Statistics, University of Michigan}

\newboolean{qedTrue}
\setboolean{qedTrue}{false}

\begin{document}\sloppy
\maketitle
\begin{abstract}
We establish sufficient conditions for the asymptotic normality of kernel density estimators, applied to causal linear random fields. Our conditions on the coefficients of linear random fields are weaker than known results, although our assumption on the bandwidth is not minimal. The proof is based on the $m$-approximation method. As a key step, we prove a central limit theorem for triangular arrays of stationary $m$-dependent random fields with unbounded $m$. We also apply a moment inequality recently established for stationary random fields.\medskip


\end{abstract}
%%%%%%%%%%%%%%%%%%%%%%%%%%%%%%%%%%%%%%%%%%%%%%%%%%%%%%%%%%%%%%%%%%%
%%%%%%%%%%%%%%%%%%%%%%%%%%%%%%%%%%%%%%%%%%%%%%%%%%%%%%%%%%%%%%%%%%%
%%%%%%%%%%%%%%%%%%%%%%%%%%%%%%%%%%%%%%%%%%%%%%%%%%%%%%%%%%%%%%%%%%%
%%%%%%%%%%%%%%%%%%%%%%%%%%%%%%%%%%%%%%%%%%%%%%%%%%%%%%%%%%%%%%%%%%%
%%%%%%%%%%%%%%%%%%%%%%%%%%%%%%%%%%%%%%%%%%%%%%%%%%%%%%%%%%%%%%%%%%%
%%%%%%%%%%%%%%%%%%%%%%%%%%%%%%%%%%%%%%%%%%%%%%%%%%%%%%%%%%%%%%%%%%%
\section{Introduction}
Let $\{X_i\}_{i\in\Zd}, d\in\N$ be a stationary zero-mean random field, such that the marginal probability density function $p(\cdot)$ exists. We are interested in the Parzen--Rosenblatt kernel density estimator of $p(x)$ in the form of 
\equh\label{eq:fn}
f_n(x) = \frac1{n^db_n}\sum_{i\in\bbr{1,n}^d} K\bpp{\frac{x-X_{i}}{b_n}}, x\in\R\,.
\eque
Throughout this paper, we assume that the kernel $K:\R\to\R$ is
a bounded Lipschitz-continuous density function, and the bandwidth $b_n$ satisfies
\equh\label{eq:bn}
b_n\to 0\qmand n^db_n\to\infty \mmas n\to\infty\,.
\eque
We also write, for $a,b\in\Z$, $\bbr{a,b} \equiv \{a,a+1,\dots,b\}$.

This problem was first considered by Rosenblatt~\cite{rosenblatt56remarks} and Parzen~\cite{parzen62estimation}, in the case that $X_i$'s are independent and identically distributed (i.i.d.) random variables: in particular, one can show the {\it consistency}
\[
\limn f_n(x) = p(x),
\]
and the {\it asymptotic normality}
\equh\label{eq:AN}
(n^db_n)^{1/2}(f_n(x)-\esp f_n(x)) \weakto \calN(0,\sigma_x^2) \qmmas n\to\infty\,,
\eque
where $\sigma_x^2 = p(x)\int K^2(s)\d s$. See for example Silverman~\cite{silverman86density} for more references on density estimation problems with i.i.d.~data.

The case that $X_i$'s are dependent, however, has presented more challenges, and we focus on establishing the asymptotic normality~\eqref{eq:AN} in this paper.
 The dependent one-dimensional case has been considered by~Robinson~\cite{robinson83nonparametric}, Castellana and Leadbetter~\cite{castellana86smoothed}, Bosq et al.~\cite{bosq99asymptotic}, Wu and Mielniczuk~\cite{wu02kernel} and Dedecker and Merlev\'ede~\cite{dedecker02necessary}, among others. 
In particular,
Wu and Mielniczuk~\cite{wu02kernel} investigated thoroughly the case when $\{X_i\}_{i\in\Z}$ is a {\it linear process}. That is, 
\[
X_i = \sif k{-\infty} a_k\epsilon_{i-k}\,, i\in\Z\,,
\]
where $\sum_ka_k^2<\infty$ and the {\it innovations} $\{\epsilon_i\}_{i\in\Z}$ are i.i.d.~random variables.
Linear processes are important in the study of stationary processes, as any stationary process can be represented as linear combinations of linear processes (the so-called {\it superlinear processes}) with martingale-difference innovations (Voln\'y et al.~\cite{volny11central}). 

%In high dimensions ($d\geq 2$), the asymptotic normality of kernel density estimations has been considered for example, under mixing-type conditions by Tran~\cite{tran90kernel} and El Machkouri~\cite{elmachkouri11asymptotic}, and under non-mixing-type conditions by Hallin et al.~\cite{hallin01density} and El Machkouri~\cite{elmachkouri11kernel}. 

%Among these results, only Hallin et al.~\cite{hallin01density} has been focused on linear random fields. See the aforementioned papers for more references.
The asymptotic normality of kernel density estimators for random fields has been considered by Tran~\cite{tran90kernel},  Hallin et al.~\cite{hallin01density}, Cheng et al.~\cite{cheng08note} and El Machkouri~\cite{elmachkouri11asymptotic,elmachkouri11kernel}, among others. 
The extension of results in one dimension to high dimensions, however, is not trivial.
As summarized in Hallin et al.~\cite{hallin01density}, `the points of $\Z^d$ do not have a natural ordering. As a result, most techniques available for one-dimensional processes do not extend to random fields.' See more references in~\cite{hallin01density} on related discussions. 

In particular, a notorious difficulty for kernel density estimation of random fields, is that one often needs more assumptions on the bandwidth $b_n$ than the {\it minimal} one~\eqref{eq:bn}. This condition is minimal in the sense that it is the natural condition for the asymptotic normality~\eqref{eq:AN} to hold when $X_i$'s are i.i.d. To the best of our knowledge, only the recent results by El Machkouri~\cite{elmachkouri11asymptotic, elmachkouri11kernel} assume no other but minimal condition~\eqref{eq:bn} on $b_n$ for dependent random fields.  

In this paper, we focus on the kernel density estimation for {\it causal} linear random fields $\indzd X$ ($d\in \N$) in form of 
\equh\label{eq:Xij}
X_{ i} = \sum_{k\in\Zd,k\ggeq \vv 0}a_{k}\epsilon_{i-k}\,, i\in\Zd,
\eque
where $\sum_{i\ggeq \vv0}a_{ i}^2<\infty$ and $\indzd\epsilon$ are i.i.d.~zero-mean random variables with finite second moments. Throughout this paper, we let `$i\ggeq k$' denote `$i_\tau\geq k_\tau$ for all $\tau = 1,\dots,d$' for $i,k\in\Zd$, and write $\vv 0 = (0,\dots,0), \vv 1=(1,\dots,1)\in\Zd$.

We provide new conditions on the coefficient $\indzd a$ such that the asymptotic normality~\eqref{eq:AN} holds (see Theorem~\ref{thm:1} below), and compare with results obtained by Hallin et al.~\cite{hallin01density} and El Machkouri~\cite{elmachkouri11kernel}.  
In both cases, our conditions are weaker on the {\it coefficients} $\indzd a$. On the other hand, our condition on the {\it bandwidth}  improves the one in~\cite{hallin01density}, but it is still stronger than the minimal one~\eqref{eq:bn} assumed in~\cite{elmachkouri11kernel}.
We do not compare our result with Cheng et al.~\cite{cheng08note}, as there is a mistake in their proof (see Remark~\ref{rem:cheng} below).

Our proof is based on the $m$-approximation approach. As we will see, to address this problem one  has to establish an $m$-approximation with {\it unbounded} $m$ ($m_n\to \infty$ as $n\to\infty$). As a key step of our approach, we establish a central limit theorem for triangular arrays of stationary $m$-dependent random fields with unbounded $m$ (Theorem~\ref{thm:mn}). 
This result improves a central limit theorem established by Heinrich~\cite{heinrich88asymptotic}.
Our $m$-approximation method is also involved with certain moment inequalities for stationary random fields (Lemma~\ref{lem:Rn}). These moment inequalities are variations of the one established in~Wang and Woodroofe~\cite{wang11new}, based on the maximal inequalities for stationary sequences ($d=1$) by Peligrad and Utev~\cite{peligrad05new} (see also~\cite{peligrad07maximal, volny07nonadapted}). In general, the $m$-approximation method has been successful in proving central limit theorems for random fields (see e.g.~Cheng et al.~\cite{cheng06central}, Wang and Woodroofe~\cite{wang11new} and El Machkouri et al.~\cite{elmachkouri11central}). In particular, El Machkouri~\cite{elmachkouri11asymptotic, elmachkouri11kernel} also established $m$-approximations with unbounded $m$, combined with  Lindeberg's method (see e.g.~Rio~\cite{rio95about} and Dedecker~\cite{dedecker98central}), to prove asymptotic normality.

At last, we point out that when the asymptotic normality~\eqref{eq:AN} holds, the random variables are often said to have {\it weak dependence}, in the sense that they behave asymptotically as i.i.d.~random variables. On the other hand, when the dependence is strong enough, the normalization for obtaining limiting distributions is of different order from $n^db_n$ in~\eqref{eq:AN}, and the asymptotic limit may be no longer Gaussian (see e.g.~Cs\"orgo and Mielniczuk~\cite{csorgo95density} for one-dimensional case). 
These two regimes are sometimes referred to as  {\it short-range dependence} and {\it long-range dependence}, respectively.
For linear processes, 
Wu and Mielniczuk~\cite{wu02kernel} addressed both short-range and long-range dependence cases. For the linear random fields, however, to the best of our knowledge, the long-range dependence case remains open. It seems that the $m$-approximation method is limited to the short-range dependence case. Therefore, the long-range dependence case is beyond the scope of this paper.

The paper is organized as follows. Our assumptions and main results are presented in Section~\ref{sec:main}. Examples and comparison with other results are provided in Section~\ref{sec:examples}. Section~\ref{sec:mn} is devoted to the central limit theorem for triangular arrays of $m$-dependent random fields. 
Section~\ref{sec:CLT} establishes asymptotic normality by $m$-approximation.
Auxiliary proofs are given in Section~\ref{sec:proofs}. 

%%%%%%%%%%%%%%%%%%%%%%%%%%%%%%%%%%%%%%%%%%%%%%%%
%%%%%%%%%%%%%%%%%%%%%%%%%%%%%%%%%%%%%%%%%%%%%%%%
%%%%%%%%%%%%%%%%%%%%%%%%%%%%%%%%%%%%%%%%%%%%%%%%
%%%%%%%%%%%%%%%%%%%%%%%%%%%%%%%%%%%%%%%%%%%%%%%%
%%%%%%%%%%%%%%%%%%%%%%%%%%%%%%%%%%%%%%%%%%%%%%%%
%%%%%%%%%%%%%%%%%%%%%%%%%%%%%%%%%%%%%%%%%%%%%%%%
%%%%%%%%%%%%%%%%%%%%%%%%%%%%%%%%%%%%%%%%%%%%%%%%
%%%%%%%%%%%%%%%%%%%%%%%%%%%%%%%%%%%%%%%%%%%%
%%%%%%%%%%%%%%%%%%%%%%%%%%%%%%%%%%%%%%%%%%%%
%%%%%%%%%%%%%%%%%%%%%%%%%%%%%%%%%%%%%%%%%%%%
%%%%%%%%%%%%%%%%%%%%%%%%%%%%%%%%%%%%%%%%%%%%
%%%%%%%%%%%%%%%%%%%%%%%%%%%%%%%%%%%%%%%%%%%%
%%%%%%%%%%%%%%%%%%%%%%%%%%%%%%%%%%%%%%%%%%%%
%%%%%%%%%%%%%%%%%%%%%%%%%%%%%%%%%%%%%%%%%%%%
%%%%%%%%%%%%%%%%%%%%%%%%%%%%%%%%%%%%%%%%%%%%
\section{Assumptions and the main result}\label{sec:main}
We first introduce our conditions. 
For each $m\in\N, i\in\Zd$, write
\equh\label{eq:Xijm}
X_{{ i},m} = \sum_{ k\in\bbr{0,m-1}^d}a_{ k}\epsilon_{ i- k}\qmand \wt X_{ i,m} = X_{ i} - \wt X_{ i,m}\,.
\eque
Let $p$, $p_m$ and $\wt p_m$ denote the probability density function of $X_{\vv 0}$, $X_{\vv 0,m}$ and $\wt X_{\vv 0,m}$, respectively. Let $p_{ i}$ and $p_{ i,m}$ denote the joint density functions of $(X_{\vv 0},X_{ i})$ and $(X_{\vv 0,m}, X_{ i,m})$, respectively. 
Our first condition is on the regularity of the density functions. Define the supremum $\wb p = \sup_x p(x)$, $\wb p_{ i} = \sup_{x,y}p_{ i}(x,y)$ and similarly $\wb p_m$ and $\wb p_{ i,m}$. 

\begin{Cond}\label{cond:D}
(i) The density functions $p$ and $\{p_m\}_{m\in\N}$ exist. They are $c_0$-Lipschitz continuous with certain constant $c_0<\infty$, independent of $m$ (i.e., $\max(|p(x)-p(y)|,|p_m(x)-p_m(y)|)\leq c_0|x-y|$). Furthermore, 
\equh\label{eq:p*}
\wb p<\infty\qmand \sup_m\wb p_m<\infty\,.
\eque
\noindent (ii) The density functions $p_{ i}$ and $p_{ i,m}$ exist for all $ i\neq \vv0, m\in\N$. Furthermore, 
\equh\label{eq:pij*}
\sup_{ i\neq \vv0} \wb p_{ i}<\infty\qmand \sup_m\sup_{ i\neq\vv0} \wb p_{ i,m}<\infty\,.
\eque
\end{Cond}

Condition~\ref{cond:D} can be satisfied, for example, by simply assuming that the probability density function $p_\epsilon$ of $\epsilon_{\vv 0}$ exists and is Lipschitz. This was assumed also in Wu and Mielniczuk~\cite{wu02kernel}.
\begin{Lem}\label{lem:pij}
If $p_\epsilon$ exists and is Lipschitz, then Condition~\ref{cond:D} holds.
\end{Lem}
The proof is deferred to Section~\ref{sec:proofs}.

Our second condition is on the decay of coefficients and bandwidth $b_n$. Define
\[
A_{ k} = \bpp{\sum_{ i\ggeq  k} a_{ i}^2}^{1/2}, k\in\Zd \mand B_{m} = \bpp{\sum_{\substack{i\in\bbr{0,\infty}^d\\ |i|_\infty\geq m}}a_{ i}^2}^{1/2}, m\in\N\,,
\]
with $|i|_\infty = \max_{\tau = 1,\dots,d}|i_\tau|$.
%Let $\indn m$ be a sequence of integers such that $m_n\to\infty$ as $n\to\infty$. 
Write
\[
\Delta_n = \sum_{k\in\bbr{1,n}^d} \frac{A_{k-\vv 1}}{\prodd \tau1dk_\tau^{1/2}}\,.
\]
\begin{Cond}\label{cond:C}
There exist a sequence of integers $\indn m$ such that $m_n\to\infty$ as $n\to\infty$, and the following limits hold:
\eqnhn\label{eq:C1}
\limn b_n^{1/2}\Delta_n & = & 0\,,\\
\label{eq:C2}%mn
\limn \frac{B_{m_n}}{b_n} & = & 0\,,\\
\limn m_n^db_n & = & 0\,,\label{eq:C3}
\\
\label{eq:C4}%mn3
\limn\frac{m_n^{d}\log^{d} n}{n^{d}b_n} & = & 0\,.
\eqnen

\end{Cond}
\begin{Thm}\label{thm:1}
If Conditions~\ref{cond:D} and~\ref{cond:C} hold and $\esp(|\epsilon_{\vv 0}|^\alpha)<\infty$ for some $\alpha>2$, then
the asymptotic normality~\eqref{eq:AN} holds.
\end{Thm}
We will prove Theorem~\ref{thm:1} in Section~\ref{sec:CLT}. We conclude this section with a few remarks.
\begin{Rem}
We briefly comment on each condition in Condition~\ref{cond:C}.
\begin{itemize}
\item[(i)] Condition~\eqref{eq:C1} is slightly weaker than
\[
\Delta_\infty\equiv \sum_{k\in\bbr{1,\infty}^d}\frac{A_{k-\vv 1}}{\prodd\tau1dk_\tau^{1/2}}<\infty\,.
\]
It was shown in~\cite{wang11new}, Corollary 1 that, the above condition implies the asymptotic normality of $\sum_{i\in\bbr{1,n}^d}[f(X_{ i}) - \esp f(X_{\vv 0})]/n^{d/2}$ for Lipschitz continuous function $f$ such that $\esp f^2(X_{\vv 0})<\infty$.
\item[(ii)] Condition~\eqref{eq:C2} implies that
\equh\label{eq:C2'}
\limn \frac{\esp|\wt X_{\vv 0,m_n}|}{b_n} = 0\,.
\eque
Indeed, Wu~\cite{wu02central}, Lemma 4 showed that for i.i.d.~zero-mean random variables $\{\epsilon_i\}_{i\in\Z}$ with $\esp(|\epsilon_{\vv 0}|^{2\vee 2p})<\infty, p>0$, 
\equh\label{eq:wu02}
\esp\bpp{\babs{\sum_{i}a_{i}\epsilon_{i}}^{2p}}\leq C\bpp{\sum_{i}a_{i}^2}^p\,.
\eque
Intuitively, $\wt X_{\vv 0,m_n}$ can be viewed as the remainder of $X_{\vv 0}$ after the $m_n$-truncation. Condition~\eqref{eq:C2'} tells that $m_n$ needs to tend to infinity fast enough, so that the central limit theorem holds. 
\item [(iii)] Conditions~\eqref{eq:C3} and~\eqref{eq:C4} are useful when we apply a central limit theorem for $m$-dependent random variables with unbounded $m$ in Proposition~\ref{prop:1} below.
\end{itemize}
\end{Rem}
Throughout this paper, let $C$ denote constants that do not depend on $ i,k,m,n,x,y$. The value of $C$ may change from line to line.

%%%%%%%%%%%%%%%%%%%%%%%%%%%%%%%
%%%%%%%%%%%%%%%%%%%%%%%%%%%%%%%
%%%%%%%%%%%%%%%%%%%%%%%%%%%%%%%
%%%%%%%%%%%%%%%%%%%%%%%%%%%%%%%
%%%%%%%%%%%%%%%%%%%%%%%%%%%%%%%
%%%%%%%%%%%%%%%%%%%%%%%%%%%%%%%
%%%%%%%%%%%%%%%%%%%%%%%%%%%%%%%
\section{Examples and discussions}\label{sec:examples}
Theorem~\ref{thm:1}, and particularly Condition~\ref{cond:C}, is not convenient to apply for concrete models. Instead, we provide a corollary for practical reason.  Write
\[
A_{[n]} = \max\{A_{n,1,\dots,1},A_{1,n,1,\dots,1},\dots,A_{1,\dots,1,n}\}\,.
\]
\begin{Coro}\label{coro:1}
Suppose $A_{[n]} \leq c_1n^{-\beta}$ and $\beta>0$, and $b_n = c_2n^{-\gamma}$. Then a sufficient condition such that Condition~\ref{cond:C} holds is 
\equh\label{eq:gammabeta}
\gamma<\frac{d\beta}{d+\beta} \qmand \beta>d\,.
\eque
Consequently, if $\esp(|\epsilon_{\vv 0}|^\alpha)<\infty$ for some $\alpha>2$, and Condition~\ref{cond:D} and~\eqref{eq:gammabeta} hold, then the asymptotic normality~\eqref{eq:AN} follows.
\end{Coro}
\begin{proof}
Assume that $m_n$ takes the form of $\sfloor{n^\delta}$. Observe that $B_{m_n}$ is of the same order of $A_{[m_n]}$ as $n\to\infty$.
Then, the limit conditions~\eqref{eq:C2},~\eqref{eq:C3} and~\eqref{eq:C4} are implied by 
\[
\limn n^{-\beta\delta+\gamma} + n^{d\delta-\gamma} + n^{\delta-1+\gamma/d} = 0\,,
\]
which is equivalent to 
$\gamma/\beta<\delta <\min\{\gamma/d,1-\gamma/d\}$.
Since $\beta>d$ implies that $\Delta_\infty <\infty$, the desired result follows.
\end{proof}
\begin{Rem}
Under the assumptions of Corollary~\ref{coro:1}, Condition~\eqref{eq:gammabeta} is very close to necessary for Condition~\ref{cond:C} to holds. Indeed, if $A_{[n]} = l(n)n^{-\beta}$ with $\limn l(n) = c_2>0$, then the same argument above yields that Condition~\ref{cond:C} is equivalent to~\eqref{eq:gammabeta}.
\end{Rem}
Below, we provide examples of coefficients so that Condition~\ref{cond:C} holds. We assume that $b_n = n^{-\gamma}$ for some $\gamma\in(0,d)$.
\begin{Example}
We compare our conditions and the ones by Hallin et al.~\cite{hallin01density}. They considered the case that $|a_{ i}|\leq C|i|_\infty^{-q},  i\ggeq \vv0$. Then, they require 
\equh\label{eq:hallin}
q>\max(d+3,2d+1/2) \qmand \limn n^db_n^{(2q-1+6d)/(2q-1-4d)}  = \infty\,.
\eque
Our condition~\eqref{eq:gammabeta} imposes weaker assumption in this case (with $b_n = n^{-\gamma}$). First, observe that 
\[
A_{n,1,\dots,1}^2 \leq B_n^2\leq  C\sif {i}n{i^{d-1} i^{-2q}} \leq Cn^{d-2q}\,.
\]
We can apply Corollary~\ref{coro:1} with $\beta = q-d/2$. Then,~\eqref{eq:gammabeta} becomes 
\equh\label{eq:hallin1}
q>\frac{3d}2\qmand \gamma<d\frac{q-d/2}{q+d/2} \,.
\eque
Thus, to establish the asymptotic normality~\eqref{eq:AN}, our condition~\eqref{eq:hallin1} is less restrictive than~\eqref{eq:hallin} on both $q$ and $\gamma$. 
\end{Example}
\begin{Example}
We compare our conditions and the ones by El Machkouri~\cite{elmachkouri11kernel}. Note that his results apply to general stationary random fields and the linear random fields are a specific case. In particular, he showed that for causal linear random fields, if 
\equh\label{eq:q}
\sum_{i\in\Zd} |i|_\infty^q |a_{ i}|<\infty 
\eque
with $q = 5d/2$, then the asymptotic normality follows.

In this case, our condition on the {\it coefficients} is weaker, requiring only $q>d$. Indeed, suppose~\eqref{eq:q} holds with some $q>0$.  
Then, to apply Corollary~\ref{coro:1}, it suffices to observe
\[
A_{n,1}^2 = \sif {i_1}n\sum_{\substack{i_2,\dots,i_d\in\N}}|a_{ i}|^2\leq Cn^{-2q}\sif {i_1}n\sum_{\substack{i_2,\dots,i_d\in\N}} |i|_\infty^{2q}|a_{ i}|^2< Cn^{-2q},
\]
and take $\beta = q$.

At the same time, our result requires $\gamma<dq/(q+d)$ for the bandwidth, in addition to the minimal one~\eqref{eq:bn} assumed in~\cite{elmachkouri11kernel}. Recall also that we assume $\esp(|\epsilon_{\vv 0}|^\alpha)<\infty$ for some $\alpha>2$, while El Machkouri's result needs only finite-second-moment assumption on $\epsilon_{\vv 0}$.
\end{Example}

\begin{Rem}
Finally, we compare our result to Wu and Mielniczuk~\cite{wu02kernel}. In the one-dimensional case, to have asymptotic normality they assume only finite variance of $\epsilon_{\vv0}$ and weaker assumption on the coefficient: 
\equh\label{eq:q0}
\sif i0 |a_i|<\infty\,.
\eque
This is weaker than our condition in one dimension (with $q>d=1$ in~\eqref{eq:q}).

Wu and Mielniczuk followed a martingale approximation approach. It remains an open question that in high dimension, whether the condition $q>d$ in~\eqref{eq:q} can be improved to match~\eqref{eq:q0} in dimension one.
\end{Rem}
%%%%%%%%%%%%%%%%%%%%%%%%%%%%%%%%%%%%%%%%%%%%
%%%%%%%%%%%%%%%%%%%%%%%%%%%%%%%%%%%%%%%%%%%%
%%%%%%%%%%%%%%%%%%%%%%%%%%%%%%%%%%%%%%%%%%%%
%%%%%%%%%%%%%%%%%%%%%%%%%%%%%%%%%%%%%%%%%%%%
%%%%%%%%%%%%%%%%%%%%%%%%%%%%%%%%%%%%%%%%%%%%
%%%%%%%%%%%%%%%%%%%%%%%%%%%%%%%%%%%%%%%%%%%%
%%%%%%%%%%%%%%%%%%%%%%%%%%%%%%%%%%%%%%%%%%%%
\section{A central limit theorem for $m$-dependent random fields}\label{sec:mn}
In this section, we prove a central limit theorem for stationary triangular arrays of $m$-dependent random fields. 
Throughout this section, let $\{Y_{n,i}:i\in\N^d\}_{n\in\N}$ denote stationary zero-mean triangular arrays. That is, for each $n$, $\{Y_{n,i}\}_{i\in\N^d}$ is stationary and $Y_{n,i}$ has zero mean.   Furthermore, we assume that $\{Y_{n,i}\}_{i\in\N^d}$ is $m_n$-dependent in the sense that $Y_{n,i}$ and $Y_{n,j}$ are independent if $|i-j|_\infty\geq m$. We provide conditions such that
\equh\label{eq:SnY}
\frac{S_n(Y)}{n^{d/2}} \equiv \frac{\sum_{i\in\bbr{1,n}^d}Y_{n,i}}{n^{d/2}} \weakto\calN(0,\sigma^2)\mmas n\to\infty.
\eque
A key condition is the following:
\equh\label{eq:Yn}
\bnn{\sum_{i\in\N^d, \vv 1\lleq i\lleq j}Y_{n,i}}_2 \leq C(j_1\cdots j_d)^{1/2}\mfa n\in \N, j\in\N^d\,.
\eque
\begin{Rem}
Inequality~\eqref{eq:Yn} has been established, under various conditions on the dependence of stationary random fields, by Dedecker~\cite{dedecker01exponential}, Wang and Woodroofe~\cite{wang11new}, and El Machkouri et al.~\cite{elmachkouri11central}, among others.
\end{Rem}
\begin{Thm}\label{thm:mn}
Suppose that there exists a constant $C$ such that~\eqref{eq:Yn} holds.
If there exists a sequence $\indn l\subset\N$, $m_n/l_n\to 0$ and $l_n/n\to 0$ as $n\to\infty$, such that 
\eqnhn\label{eq:LF1}
& & \limn\frac1{l_n^d}\esp\bbb{\bpp{\sum_{k\in\bbr{1,l_n}^d}Y_{n,k}}^2} = \sigma^2\,,\\
\label{eq:LF2}
& & \limn\frac1{l_n^d}\esp\bbb{\bpp{\sum_{k\in\bbr{1,l_n}^d}Y_{n,k}}^2\ind\bccbb{\babs{\sum_{k\in\bbr{1,l_n}^d}Y_{n,k}}> n^{d/2}\epsilon}} = 0\,,
\eqnen
for all $\epsilon>0$, 
then~\eqref{eq:SnY} holds.
\end{Thm}
\begin{proof}
Consider partial sums over big blocks of size $l_n^d$, denoted by
\[
\eta_{n,k} = \sum_{i\in\bbr{1,l_n}^d}Y_{n,i+k(l_n+m_n)}\,, k\in \N^d.
\]
In this way, for each $n\in\N$, $\{\eta_{n,k}\}_{k\in\N^d}$ are i.i.d., as we separate neighboring blocks by distance $m_n$, and $\{Y_{n,i}\}_{i\in\Zd}$ are $m_n$-dependent.
Set 
\[
S_n(\eta) = \sum_{k\in\bbr{0,\floor{n/(l_n+m_n)}-1}^d}\eta_{n,k}\,, n\in\N\,.
\]
Then,~\eqref{eq:Yn} implies that
\[
\bnn{\frac{S_n(Y)}{n^{d/2}} - \frac{S_n(\eta)}{n^{d/2}}}_2\to 0 \mmas n\to\infty\,.
\]
To see this, for the sake of simplicity, we consider the case $n/(l_n+m_n) = \floor{n/(l_n+m_n)}$. Indeed, by the triangular inequality, the left-hand side above can be bounded by sums in form of $\nn{\sum_{i\in B}Y_{n,i}}_2/n^{d/2}$, where $B$ can be a rectangle of size $n^{d-r}m_n^r$ with $r\in\{1,\dots,d-1\}$. 
Focusing on the dominant term with $r=1$, we then bound the left-hand side above by $C(n/(l_n+m_n))^{1/2}(n^{d-1}m_n)^{1/2}/n^{d/2} = Cm_n^{1/2}/(l_n+m_n)^{1/2}\to \infty$ as $n\to\infty$.

As a consequence, it suffices to show $S_n(\eta)/n^{d/2}\weakto \calN(0,\sigma^2)$. This, under conditions~\eqref{eq:LF1} and~\eqref{eq:LF2},  follows from the standard central limit theorem for triangular arrays of independent random variables (see e.g.~\cite{durrett96probability}, Chapter 2, Theorem 4.5).
\end{proof}
\begin{Rem}
Central limit theorems for $m_n$-dependent random fields has been considered by Heinrich~\cite{heinrich88asymptotic}. His result has been recently applied, with $m_n = m$ fixed, by El Machkouri et al.~\cite{elmachkouri11central} to establish a central limit theorem for stationary random fields. 

Our application requires us to take $m_n\to\infty$ (see Remark~\ref{rem:cheng} below). In this case our condition in Theorem~\ref{thm:mn} is weaker than Heinrich's. In particular, he assumed
\equh\label{eq:heinrich}
\limn \frac{m_n^{2d}}{n^d}\sum_{i\in\bbr{1,n}^d} \esp\bpp{{Y_{n,i}^2}\indd{|Y_{n,i}|>\epsilon n^{d/2}m_n^{-2d}}} = 0\,,\mfa \epsilon>0\,.
\eque
This is stronger than~\eqref{eq:LF2}.

\end{Rem}

%%%%%%%%%%%%%%%%%%%%%%%%%%%%%%%%%%%%%%%%%%%%%%
%%%%%%%%%%%%%%%%%%%%%%%%%%%%%%%%%%%%%%%%%%%%%%
%%%%%%%%%%%%%%%%%%%%%%%%%%%%%%%%%%%%%%%%%%%%%%
%%%%%%%%%%%%%%%%%%%%%%%%%%%%%%%%%%%%%%%%%%%%%%
%%%%%%%%%%%%%%%%%%%%%%%%%%%%%%%%%%%%%%%%%%%%%%
%%%%%%%%%%%%%%%%%%%%%%%%%%%%%%%%%%%%%%%%%%%%%%
%%%%%%%%%%%%%%%%%%%%%%%%%%%%%%%%%%%%%%%%%%%%%%
%%%%%%%%%%%%%%%%%%%%%%%%%%%%%%%%%%%%%%%%%%%%%%
%%%%%%%%%%%%%%%%%%%%%%%%%%%%%%%%%%%%%%%%%%%%%%
%%%%%%%%%%%%%%%%%%%%%%%%%%%%%%%%%%%%%%%%%%%%%%
%%%%%%%%%%%%%%%%%%%%%%%%%%%%%%%%%%%%%%%%%%%%%%
%%%%%%%%%%%%%%%%%%%%%%%%%%%%%%%%%%%%%%%%%%%%%%

\section{Asymptotic normality by $m$-approximation}\label{sec:CLT}
In this section, we prove Theorem~\ref{thm:1} by an $m$-approximation argument. Fix $x\in\R$ and write
\[
Z_{n, i} = \frac1{\sqrt {b_n}}K\bpp{\frac{x-X_{ i}}{b_n}}\qmand \zeta_{n, i} = \frac1{\sqrt {b_n}}K\bpp{\frac{x-X_{ i,m_n}}{b_n}}, i\in\Zd\,.
\]
In this way, $\{\zeta_{n, i}\}_{ i\in\Zd}$ are $m_n$-dependent. We will use $\{\zeta_{n, i}: i\in\Zd\}_{n\in\N}$ to approximate $\{Z_{n, i}: i\in\Zd\}_{n\in\N}$. We also write $\wb Z_{n, i} = Z_{n, i}-\esp Z_{n, i}$ and $\wb \zeta_{n, i} = \zeta_{n, i} - \esp \zeta_{n, i}$. 
Setting  
\[
S_{n}(\wb \zeta) = \sum_{i\in\bbr{1,n}^d} \wb \zeta_{n, i} \qmand 
S_{n}(\wb Z-\wb\zeta) = \sum_{i\in\bbr{1,n}^d} (\wb Z_{n, i} - \wb \zeta_{n, i})\,,
\]
we decompose
\equh\label{eq:B}
(n^db_n)^{1/2}(f_n(x)-\esp f_n(x)) = \frac{S_{n}(\wb\zeta)}{n^{d/2}} + \frac{S_{n}(\wb Z-\wb\zeta)}{n^{d/2}}\,.
\eque
To prove Theorem~\ref{thm:1}, it suffices to establish the following two results.
\begin{Prop}\label{prop:1}Under Condition~\ref{cond:D} and~\eqref{eq:C1},~\eqref{eq:C3},~\eqref{eq:C4} of Condition~\ref{cond:C},
\equh\label{eq:MnB}
\frac{S_{n}(\wb\zeta)}{n^{d/2}}\weakto \calN(0,\sigma_{x}^2)\,.
\eque
\end{Prop}
\begin{Prop}\label{prop:2}
Under Condition~\ref{cond:D} and~\eqref{eq:C1},~\eqref{eq:C2} of Condition~\ref{cond:C},
\equh\label{eq:limRnB}
\frac{{S_{n}(\wb Z-\wb\zeta)}}{n^{d/2}}\probato 0\,.
\eque
\end{Prop}

To prove the above two propositions, a key step is to establish the following moment inequalities.
\begin{Lem}\label{lem:Rn}
There exists a constant $C>0$, such that for all $n\in\N$, 
\equh\label{eq:Rn}
{\nn {S_n(\wb Z-\wb\zeta)}_2}\\
\leq 
Cn^{d/2}\bpp{\nn{\wb Z_{n,\vv 0}-\wb\zeta_{n,\vv 0}}_2 + b_n^{1/2}\Delta_n}\,.
\eque
In addition, if $\esp(|\epsilon_{\vv0}|^\alpha)<\infty$ for some $\alpha\geq 2$, then 
\equh\label{eq:momentp}
\bnn{\sum_{i\in\N^d,\vv1\lleq i\lleq j}\wb\zeta_{n,i}}_\alpha\leq C(j_1\cdots j_d)^{1/2}\bpp{\nn{\wb\zeta_{n,\vv 0}}_\alpha + b_n^{1/2}\Delta_n}\,,\mfa j\in\N^d.
\eque

\end{Lem}
These inequalities are consequences of the moment inequality recently established in Wang and Woodroofe~\cite{wang11new}. The proof is deferred to Section~\ref{sec:proofs}.

\begin{proof}[Proof of Proposition~\ref{prop:1}]
Observe that $S_{n}(\wb\zeta)/n^{d/2}$ is a partial sum of $m_n$-dependent random fields and we apply Theorem~\ref{thm:mn}. Observe that since $\nn{\wb\zeta_{n,\vv 0}}_2\to \sigma_x$ as $n\to\infty$,~\eqref{eq:momentp} with $\alpha=2$ and assumption~\eqref{eq:C1} entail~\eqref{eq:Yn}. Thus, to prove~\eqref{eq:MnB}, it suffices to show, for $l_n = m_n\log n$, 
\equh\label{eq:mCLT1}
\limn \frac1{l_n^d}\esp\bbb{\bpp{\sum_{i\in\bbr{1,l_n}^d} \wb \zeta_{n, i}}^2} = \sigma_x^2\,,
\eque
and, writing $\xi_n = \sum_{i\in\bbr{1,l_n}^d}\wb\zeta_{n, i}$, 
\equh\label{eq:mCLT2}
\limn \frac1{l_n^d}\esp\bpp{\xi_n^2\indd{|\xi_n|>n^{d/2}\epsilon}} = 0\,,\mfa \epsilon>0\,.
\eque
By standard calculation, under~\eqref{eq:pij*} of Condition~\ref{cond:D}, for all $n\in\N$ and $i\neq \vv 0$,
\[
 |\esp(\wb\zeta_{n,\vv 0}\wb\zeta_{n, i})| \leq C\wb p_{i,m_n}b_n\leq Cb_n\,.
\]
Therefore, 
\begin{multline*}
\babs{\frac1{l_n^d}\esp\bbb{\bpp{\sum_{i\in\bbr{1,l_n}^d}\wb\zeta_{n, i}}^2} - \esp\wb\zeta_{n,\vv 0}^2} \\
\leq 2\sum_{i\in\bbr{-m_n,m_n}^d}\sabs{\esp(\wb\zeta_{n,\vv 0}\wb\zeta_{n, i})}\indd{ i\neq \vv 0}
\leq Cm_n^db_n\,.
\end{multline*}
Thus, assumption~\eqref{eq:C3} entails~\eqref{eq:mCLT1}.
To prove~\eqref{eq:mCLT2}, observe that 
\[
\esp(\xi_n^2\indd{|\xi_n|>n^{d/2}\epsilon}) \leq \nn{\xi_n}_\alpha^2\proba(|\xi_n|>n^{d/2}\epsilon)^{(\alpha-2)/\alpha} \leq \nn{\xi_n}_\alpha^2\bpp{\frac{\nn{\xi_n}_2^2}{n^d\epsilon^2}}^{(\alpha-2)/\alpha}\,.
\]
This time,~\eqref{eq:momentp} and~\eqref{eq:C1} yield $\nn{\xi_n}_2\leq Cl_n^{d/2}$. For $\alpha>2$, observe that, since $K$ is bounded,  
\[
\nn{\wb\zeta_{n,\vv 0}}_\alpha = \spp{\esp|\wb\zeta_{n,\vv 0}|^\alpha}^{1/\alpha} \leq \bpp{\frac C{b_n^{(\alpha-2)/2}}\nn{\wb\zeta_{n,\vv 0}}_2^2}^{1/\alpha}\leq Cb_n^{-(\alpha-2)/(2\alpha)}\,.
\]
So, $\nn{\xi_n}_\alpha^2\leq Cl_n^db_n^{-(\alpha-2)/\alpha}$.
To sum up, we have obtained that
\[
\frac1{l_n^d}\esp(\xi_n^2\indd{|\xi_n|>n^{d/2}\epsilon}) \leq C{\bpp{\frac{l_n^d}{n^db_n}}^{(\alpha-2)/\alpha}}\,.
\]
Now,~\eqref{eq:C4} entails~\eqref{eq:mCLT2}.
\end{proof}

\begin{proof}[Proof of Proposition~\ref{prop:2}]
To obtain the desired result, it suffices to combine~\eqref{eq:Rn}, assumptions~\eqref{eq:C1} and~\eqref{eq:C2} and Lemma~\ref{lem:mn} below.
\end{proof}

\begin{Lem}\label{lem:mn}
Under the assumption of Condition~\ref{cond:D}, there exists a constant $C$, such that for all $n\in\N$, 
\equh\label{eq:mn0}
\nn{\wb\zeta_{n,\vv 0}-\wb Z_{n,\vv 0}}_2 \leq C\bbb{\bpp{\frac{B_{m_n}}{b_n}}^{1/2} + b_n^{1/2}}\,.
\eque
\end{Lem}
The proof is deferred to Section~\ref{sec:proofs}.
\begin{Rem}\label{rem:cheng}Cheng et al.~\cite{cheng08note} also considered the asymptotic normality of kernel density estimators for linear random fields. Their approach combines an $m$-approximation with a martingale approximation by defining an appropriate filtration in $\Z^2$. However, there is a mistake in Lemma 2 therein. In our notation, they claimed that, instead of~\eqref{eq:mn0}, there exists a constant $C$, such that (in the case $d=2$)
\equh\label{eq:cheng}
\nn{\zeta_{n,\vv 0} - Z_{n,\vv 0}}_2^2\leq Cb_n \mwith m_n\equiv m\,.
\eque
To see that~\eqref{eq:cheng} is not true, observe that
\[
\nn{\zeta_{n,\vv 0}-Z_{n,\vv 0}}^2_2 = \esp \zeta_{n,\vv 0}^2 + \esp Z_{n,\vv 0}^2 - 2\esp(\zeta_{n,\vv 0}Z_{n,\vv 0}).
\]
By standard calculations, $\limn \esp Z_{n,\vv 0}^2= p(x)\int K(s)^2\d s$, and if $m_n\equiv m$, then $\limn\esp\zeta_{n,\vv 0}^2 = p_m(x)\int K(s)^2\d s$ and $
\limn\esp(\zeta_{n,\vv 0}Z_{n,\vv 0})=0$. 
Therefore, the left-hand side of~\eqref{eq:cheng} has a strictly positive limit as $n\to\infty$ (unless $p(x)=p_m(x)=0$), thus a contradiction.
Their approach might still work by adapting an $m$-approximation with unbounded $m$, although it is not clear to us what conditions it would lead to.

%Instead of~\eqref{eq:cheng}, we establish the following control of $\nn{\wb\zeta_{n,\vv 0}-\wb Z_{n,\vv 0}}_2$.
\end{Rem}

%%%%%%%%%%%%%%%%%%%%%%%%%%%%%%%%%%%%%%%%%%%%%%
%%%%%%%%%%%%%%%%%%%%%%%%%%%%%%%%%%%%%%%%%%%%%%
%%%%%%%%%%%%%%%%%%%%%%%%%%%%%%%%%%%%%%%%%%%%%%

\section{Proofs}\label{sec:proofs}
\begin{proof}[Proof of Lemma~\ref{lem:pij}]
(i) The existence and Lipschitz continuity of $p$ and $p_m$ have been proved by Wu and Mielniczuk~\cite{wu02kernel}, Lemma 1. 
To prove~\eqref{eq:p*}, observe that
\eqnhn
|p_m(y) - p(y)| & \leq & \int|p_m(y)-p_m(y-x)|\wt p_{m}(x)\d x\nonumber\\
& \leq & C\int|x|\wt p_{m}(x)\d x =C\esp|\wt X_{\vv 0,m}|\,.\label{eq:pm-p}
\eqnen
This entails that $p_m(x)\to p(x)$ uniformly for $x\in\R$ as $m\to\infty$. Therefore,~\eqref{eq:p*} holds.\medskip

\noindent(ii) Fix $i\in\Zd\setminus\{\vv0\}$ and let $F_{ i}$ denote the joint distribution function of $(X_{\vv 0}, X_{ i})$. For the sake of simplicity, we prove the case of $a_{\vv 0} = 1$. Write $R = X_{\vv 0} - \epsilon_{\vv 0}$ and $R_{ i} = X_{ i} - \epsilon_{ i} - a_{ i}\epsilon_{\vv 0}$. Now, $R$ and $R_{ i}$ are dependent random variables. First, we show that
\equh\label{eq:pij}
p_{ i}(x,y) \equiv \frac{\partial^2}{\partial x\partial y} F_{ i}(x,y) =  \esp [p_\epsilon(x-R)p_\epsilon(y-R_{ i}-a_{ i}x)]\,.
\eque
Indeed, 
\eqnhn
F_{ i}(x,y) & = & 
\proba(X_{\vv 0}\leq x, X_{ i}\leq y)\nonumber\\
 & = & \proba(\epsilon_{\vv 0} + R\leq x, \epsilon_{ i} + a_{ i}\epsilon_{\vv 0} + R_{ i}\leq y)\nonumber\\
& = & \esp \Phi_{ i}(x-R,y-R_{ i})\,,\label{eq:differentiation}
\eqnen
with, letting $F_\epsilon$ denote the cumulative distribution function of $\epsilon_{\vv0}$,
\[
\Phi_{ i}(x,y) = \int_{-\infty}^xF_\epsilon(y-a_{ i}x\pp)F_\epsilon(\d x\pp)\,.
\]
Differentiating~\eqref{eq:differentiation} yields~\eqref{eq:pij} (see e.g.~\cite{durrett96probability}, Appendix A.9 on the validation of exchange of differentiation and expectation).

Next, we prove~\eqref{eq:pij*} by establishing the following two steps:
\equh\label{eq:convergence_ij}
\lim_{|i|_\infty\to\infty}\sup_{x,y}|p_{ i}(x,y) - p(x)p(y-a_{ i}x)| = 0\,,
\eque
and
\equh\label{eq:uniform_m}
\limm\sup_{x,y, i}|p_{ i}(x,y) - p_{ i,m}(x,y)| = 0\,.
\eque
Then,~\eqref{eq:convergence_ij} implies the first part of~\eqref{eq:pij*}, and the two limits imply the second part.

To prove~\eqref{eq:convergence_ij}, set 
\[
\wt D_{ i} = \esp(R_{ i}\mid\sigma(\epsilon_{k}:k\lleq\vv 0)) \qmand D_{ i} = R_{ i} - \wt D_{ i}\,, i\in\Zd.
\]
By definition, $D_{ i}$ and $R$ are independent.
Introducing an intermediate term $\esp[p_\epsilon(x-R)p_\epsilon(y-D_{ i}-a_{ i}x)] = p(x)\esp p_\epsilon(y-D_{ i}-a_{ i}x)$, we then bound $|p_{ i}(x,y) - p(x)p(y-a_{ i}x)|\leq \Psi_1 + \Psi_2$ with, under the assumption that $p_\epsilon$ is bounded and Lipschitz,
\eqnh
\Psi_1 & = & |p_{ i}(x,y) - \esp [p_\epsilon (x-R) p_\epsilon(y-D_{ i}-a_{ i}x)]|\\
& \leq & \esp[ p_\epsilon(x-R)|R_{ i} - D_{ i}|]\leq C\esp|\wt D_{ i}|,
\eqne
and
\eqnh
\Psi_2 & = & |p(x)p(y-a_{ i}x) - \esp [p_\epsilon (x-R) p_\epsilon(y-D_{ i}-a_{ i}x)]|\\
& \leq & p(x)\esp|p_\epsilon(y-a_ix-R_i+a_i\epsilon_0) - p_\epsilon(y-D_i-a_ix)|\\
& \leq & C(\esp|\wt D_{ i}|+|a_i|)\,.
\eqne
By~\eqref{eq:wu02}, $|p_{ i}(x,y) - p(x)p(y-a_{ i}x)|\to 0$ as $|i|_\infty\to\infty$. 

To prove~\eqref{eq:uniform_m}, 
define $R_m = X_{\vv 0,m} - \epsilon_{\vv 0}$ and $R_{ i,m} = X_{ i,m} - \epsilon_{ i} - a_{ i}\epsilon_{\vv 0}\indd{|i|_\infty<m}$. Then, similarly as~\eqref{eq:pij}, one has
\[
p_{ i,m}(x,y) = \esp [p_\epsilon(x-R_m)p_\epsilon(y-a_{ i}x\indd{|i|_\infty<m}-R_{ i,m})]\,.
\]
Introducing an intermediate term $\esp[p_\epsilon(x-R)p_\epsilon(y-a_{ i}x\indd{|i|_\infty<m}-R_{ i,m})]$, we obtain that
\eqnh
& & |p_{ i,m}(x,y) - p_{ i}(x,y)|\\
& \leq & \esp\sbb{p_\epsilon(x-R)(|a_{ i}x|\indd{|i|_\infty\geq m} + |R_{ i}-R_{ i,m}|)} + C\esp|R-R_m|\\
& \leq & C(|x|p(x)|a_{ i}|\indd{|i|_\infty\geq m} + |R-R_m|+|R_{ i}-R_{ i,m}|)\,.
\eqne
Clearly, since $X_{\vv 0}$ has finite second moment and $p$ is bounded and Lipschitz, $\sup_x|x|p(x)<\infty$. The summability assumption on $a_{ i}$ implies that $\limm\sup_{|i|_\infty\geq m}|a_{ i}| = 0$, and $\sup_{ i}(|R-R_m|+|R_{ i}-R_{ i,m}|) \to 0$ as $m\to\infty$ (recall~\eqref{eq:wu02}). Therefore, we have thus proved~\eqref{eq:uniform_m}.
\end{proof}

\begin{proof}[Proof of Lemma~\ref{lem:Rn}]
We only prove~\eqref{eq:Rn}. The proof of~\eqref{eq:momentp} is similar. By~\cite{wang11new}, Corollary 2, there exists a constant $C$, such that
\equh\label{eq:filF}
\frac{\nn {S_n(\wb Z-\wb\zeta)}_2}n\leq C\sum_{k\in\bbr{1,n}^d} \frac{\nn{\esp(\wb Z_{n,k} - \wb \zeta_{n,k}\mid \filF_{\vv1})}_2}{\prodd\tau1dk_\tau^{1/2}}\,,
\eque
where $\filF_{\vv1} = \sigma(\epsilon_k:k\in\Zd,k\lleq \vv1)$.
By the definition of $\wb \zeta_{n, i}$,~\eqref{eq:filF} equals (up to the multiplicative constant $C$),
\eqnh
& & \sum_{k\in\bbr{1,n}^d\setminus\bbr{1,m_n}^d}\frac{\nn{\esp(\wb Z_{n,k}\mid \filF_{\vv1})}_2}{\prodd\tau1dk_\tau^{1/2}}
+ \sum_{k\in\bbr{1,m_n}^d}\frac{\nn{\esp(\wb Z_{n,k} - \wb \zeta_{n,k}\mid \filF_{\vv1})}_2}{\prodd\tau1dk_\tau^{1/2}}\\
& \leq & \nn{\wb Z_{n,\vv 0}-\wb\zeta_{n,\vv 0}}_2 + \sum_{k\in\bbr{1,n}^d}\frac{\nn{\esp(\wb Z_{n,k}\mid \filF_{\vv1})}_2}{\prodd\tau1dk_\tau^{1/2}} + \sum_{k\in\bbr{1,m_n}^d}\frac{\nn{\esp(\wb \zeta_{n,k}\mid \filF_{\vv1})}_2}{\prodd\tau1dk_\tau^{1/2}}\\
& \leq & C \bpp{\nn{\wb Z_{n,\vv 0}-\wb\zeta_{n,\vv 0}}_2 + b_n^{1/2}\sum_{k\in\bbr{1,n}^d}\frac{A_{k-\vv1}}{\prodd \tau1dk_\tau^{1/2}}}\,,
\eqne
where the last inequality follows from Lemma~\ref{lem:1} below.
\end{proof}
\begin{Lem}\label{lem:1} Suppose that in addition to Condition~\ref{cond:D}, $\esp(|\epsilon_{\vv0}|^\alpha)<\infty$ for some $\alpha\geq 2$. For all $k\in\N^d, k\neq 1$, 

\eqnhn
\nn{\esp(\wb Z_{n,k}\mid\filF_{\vv1})}_\alpha & \leq & Cb_n^{1/2}A_{k-\vv1}\,,\label{eq:wbZ}\\
\nn{\esp(\wb \zeta_{n,k}\mid\filF_{\vv1})}_\alpha & \leq &  Cb_n^{1/2}A_{k-\vv1}\,.\label{eq:wbzeta}
\eqnen
\end{Lem}
\begin{proof}[Proof of Lemma~\ref{lem:1}]
First, we control $\nn{\esp(\wb Z_{n,k} \mid \filF_{\vv1})}_\alpha$. For each $k\in\Zd$, introduce the notation
\equh\label{eq:gamma}
\Gamma(k)\defe\{i\in\Zd:i\lleq k\}\,,
\eque
and write
\[
X_{k} = \sum_{i\in\Gamma(k)}a_{k-i}\epsilon_{i} = \bpp{\sum_{i\in\Gamma(\vv 1)} + \sum_{i\in\Gamma(k)\setminus\Gamma(\vv1)}}a_{k-i}\epsilon_{i}=: D_{k} + T_{k}\,.
\]
For the sake of simplicity, write $D\equiv D_{k}, T \equiv T_{k}$, and, given a random variable $Y$, let $\esp_Y(\cdot) \equiv \esp(\cdot\mid Y)$ denote the conditional expectation given the $\sigma$-algebra generated by $Y$. 
Since $k\ggeq \vv1, k\neq \vv1$, $T_{k}$ is a non-degenerate random variable.
Then,
\[
\esp(\wb Z_{n,k}\mid\filF_{\vv 1}) = \frac1{\sqrt {b_n}}\bbb{\esp_DK\bpp{\frac{x-D-T}{b_n}} - \esp K\bpp{\frac{x-D-T}{b_n}}}\,.
\]
Let $\wt D$ be a copy of $D$, independent of $D$ and $T$. Then, the above identity becomes, letting $p_T$ denote the density of $T$,
\begin{multline*}
\frac1{\sqrt{b_n}}\esp_D\esp_{D,\wt D}\bbb{K\bpp{\frac{x-D-T}{b_n}} - K\bpp{\frac{x-\wt D-T}{b_n}}}\\
= b_n^{1/2}\esp_D\int K(t)\bbb{p_T(x-b_nt-D)-p_T(x-b_nt-\wt D)}\d t\,.
\end{multline*}
Since $p_T$ is Lipschitz, the absolute value of the above term is bounded by $C\int|K(s)|\d sb_n^{1/2} \esp_D|D-\wt D|$, almost surely.
(Here $p_T$ depends on $k,n$, but one can show that the Lipschitz constant can be chosen independently from $k,n$. See e.g.~\cite{wu02central}, Lemma 1.)
To sum up, we have
\[
\nn{\esp(\wb Z_{n,k}\mid\filF_{\vv 1})}_\alpha \leq Cb_n^{1/2}\nn{\esp_D|D-\wt D|}_\alpha\leq Cb_n^{1/2}\nn D_\alpha \leq Cb_n^{1/2}A_{k-\vv 1},
\]
where the last inequality follows from~\eqref{eq:wu02}. 
We have thus proved~\eqref{eq:wbZ}. To prove~\eqref{eq:wbzeta}, a similar argument yields 
$\nn{\esp(\wb\zeta_{n,k}\mid\filF_{\vv1})}_\alpha \leq Cb_n^{1/2}A_{k,m_n}$ with $A_{k,m_n} = (\sum_{i\in\bbr{0,m_n-1}^d, i\ggeq k-\vv1}a_{ i}^2)^{1/2}\leq A_{k-\vv 1}$.
\end{proof}

\begin{proof}[Proof of Lemma~\ref{lem:mn}]
For random variables $Z_{n,\vv 0}, \wb Z_{n,\vv 0}, \zeta_{n,\vv 0}, \wb\zeta_{n,\vv 0}$, we replace the index `$n,\vv 0$' by `$n$' for the sake of simplicity.
First observe that 
\[
(\esp Z_n)^2 + (\esp \zeta_n)^2\leq C(\wb p^2b_n+\wb p_{m_n}^2b_n)\leq Cb_n\,,
\]
where the last step we applied~\eqref{eq:pij*}. Then,
\eqnhn
|\esp(\wb\zeta_n^2 - \wb Z_{n}^2)| & \leq & \babs{\int K^2(y)[p_{m_n}(x-b_ny) - p(x-b_ny)]\d y} + C b_n \nonumber\\
& \leq & \sup_y |p_{m_n}(y)-p(y)|\int K^2(s)\d s + Cb_n\,.\nonumber\\
& \leq & C(B_{m_n}+b_n)\,,\label{eq:supy}
\eqnen
where the last inequality follows from~\eqref{eq:pm-p}.
Next, write
\equh\label{eq:dif}
\nn{\wb\zeta_n - \wb Z_n}_2^2 = \esp\wb Z^2_n - \esp \wb\zeta^2_n + 2(\esp \wb\zeta^2_n-\esp(\wb Z_n\wb\zeta_n))\,.
\eque
For the last term on the right-hand side of~\eqref{eq:dif}, observe that $\esp(\wb Z_n\wb\zeta_n) = \esp (Z_n\zeta_n) - \esp Z_n\esp \zeta_n = \esp(Z_n\zeta_n) + O(\wb p_{m_n}b_n)$. 
We claim that $\esp(Z_n\zeta_n)$ is very close to $\esp\zeta_n^2$, under our restriction on the choice of $m_n$. Indeed, 
\equh\label{eq:znzetan}
\sabs{\esp(Z_n\zeta_n) - \esp\zeta_n^2} \equiv \babs{\esp (Z_n\zeta_n) - \int K^2(y)p_{m_n}(x-b_ny)\d y}\,,
\eque
and, 
\eqnh
\esp(Z_n\zeta_n) & = & \iint\frac1{b_n}K\bpp{\frac{x-y-z}{b_n}}K\bpp{\frac{x-y}{b_n}}p_{m_n}(y)\wt p_{m_n}(z)\d y\d z\\
& = & \int K(y)\esp K\bpp{y-\frac{\wt X_{\vv 0,m_n}}{b_n}}p_{m_n}(x-b_ny)\d y\,.
\eqne
Therefore,~\eqref{eq:znzetan}
can be bounded by, since $K$ is Lipschitz,
\begin{multline*}
\int |K(y)|\esp\babs{K\bpp{y-\frac{\wt X_{\vv 0,m_n}}{b_n}} - K(y)}p_{m_n}(x-b_ny)\d y\\
\leq \frac{\esp|\wt X_{\vv 0,m_n}|}{b_n}\int|K(y)|p_{m_n}(x-b_ny)\d y\,,
\end{multline*}
and $\esp|\wt X_{\vv 0,m_n}|\leq CB_{m_n}$ by~\eqref{eq:wu02}. 
To sum up, we have thus shown that (recall that $b_n\downarrow 0$, whence $B_{m_n}$ is dominated by $B_{m_n}/b_n$), under~\eqref{eq:p*},
\[
\nn{\wb\zeta_n - \wb Z_n}_2^2 \leq C\bpp{\frac{B_{m_n}}{b_n} + b_n}\,.
\]

\end{proof}

%%%%%%%%%%%%%%%%%%%%%%%%%%%%%%%%%%%%%%%%%%%
%%%%%%%%%%%%%%%%%%%%%%%%%%%%%%%%%%%%%%%%%%%
%%%%%%%%%%%%%%%%%%%%%%%%%%%%%%%%%%%%%%%%%%%
%%%%%%%%%%%%%%%%%%%%%%%%%%%%%%%%%%%%%%%%%%%
%%%%%%%%%%%%%%%%%%%%%%%%%%%%%%%%%%%%%%%%%%%
%%%%%%%%%%%%%%%%%%%%%%%%%%%%%%%%%%%%%%%%%%%
%%%%%%%%%%%%%%%%%%%%%%%%%%%%%%%%%%%%%%%%%%%

%\bibliographystyle{abbrv}
%\bibliography{../include/references}

\def\cprime{$'$} \def\polhk#1{\setbox0=\hbox{#1}{\ooalign{\hidewidth
  \lower1.5ex\hbox{`}\hidewidth\crcr\unhbox0}}}

\end{document}